\documentclass[12 pt]{amsart}
\usepackage{amsmath}
\usepackage{amssymb}
\usepackage{amsthm}
\usepackage{amsfonts}
\usepackage{graphicx}
\usepackage{graphics}
\usepackage{epsfig}
\usepackage{color}
\usepackage{mathrsfs}
\usepackage{enumerate}

\newtheorem{thm}{Theorem}[section]
\newtheorem{cor}[thm]{Corollary}
\newtheorem{lem}[thm]{Lemma}
\theoremstyle{definition}
\newtheorem{prob}{Problem}
\theoremstyle{remark}
\newtheorem{rem}[thm]{Remark}

\numberwithin{equation}{section} \numberwithin{figure}{section}

\newcommand{\TT}{{\mathbb T}}

\begin{document}
\bibliographystyle{alpha}

\title[Cyclicity in RKHS]{Cyclicity in reproducing kernel Hilbert spaces of analytic functions}

\author[Fricain]{Emmanuel Fricain}
\address{Laboratoire Paul Painlev\'e, UFR de Math\'ematiques, B\^{a}timent M2, Universit\'e des Sciences et Technologies Lille 1, 59 655 Villeneuve d'Ascq C\'edex, France.}
\email{emmanuel.fricain@math.univ-lille1.fr}

\author[Mashreghi]{Javad Mashreghi}
\address{D\'epartament de Mathematiques et de Statistique, Universit\'e Laval, Qu\'ebec, QC, G1K 7P4, Canada.}
\email{javad.mashreghi@mat.ulaval.ca}

\author[Seco]{Daniel Seco}
\address{Mathematics Institute, Zeeman Building, University of Warwick, Coventry CV4 7AL, UK.} \email{d.seco@warwick.ac.uk}

\thanks{For this work, we were supported by grants from Labex CEMPI (ANR-11-LABX-0007-01), NSERC (100756), MEC/MICINN Project MTM2008-00145 and ERC Grant 2011-ADG-20110209 from EU programme FP2007-2013.}

\date{\today}

\begin{abstract}
We introduce a large family of reproducing kernel Hilbert spaces $\mathcal{H} \subset \mbox{Hol}(\mathbb{D})$, which include the classical Dirichlet-type spaces $\mathcal{D}_\alpha$, by requiring normalized monomials to form a Riesz basis for $\mathcal{H}$. Then, after precisely evaluating the $n$-th optimal norm and the $n$-th approximant of $f(z)=1-z$,  we completely characterize the cyclicity of functions in $\mbox{Hol}(\overline{\mathbb{D}})$ with respect to the forward shift.
\end{abstract}

\maketitle

\section{Introduction}

\subsection{Overview}
In this note, we study the cyclicity problem with respect to the
forward shift operator acting on a rather general reproducing kernel
Hilbert space (RKHS) which includes in particular the Hardy space,
the Dirichlet space and the Bergman space. This problem has a long
outstanding history and many efforts have been dedicated to solving
it in various RKHS. In his pioneering work \cite{B1}, Beurling
showed that cyclicity in the Hardy space $H^2$ is equivalent to
being outer. Brown and Shields studied the cyclicity in the
Dirichlet space for the polynomials that do not have zeros inside
the disc, but that do have them on its boundary. Such functions are
cyclic in $D_\alpha$ if and only if $\alpha \leq 1$. They also
proved that the set of zeros (in radial sense) of cyclic functions
in the Dirichlet space has zero logarithmic capacity and this led
them to ask whether any outer function with this property is cyclic
\cite{BS1}. This problem is still open although there has been
relevant contributions to the topic by a number of authors; e.g. see
\cite{B2, BC1, EFKR1, EFKR2, HS1, RS1} and the survey papers
\cite{ARSW1, R1}.

In the recent paper \cite{BCLSS1}, a method is given to find the sequence of the $n$-th optimal approximant $(p_n^*)_{n \geq 0}$, and the $n$-th optimal norm for the Dirichlet space. The first aim of the present paper is to study these concepts in the more general context of RKHS introduced above.

\subsection{Reproducing kernel Hilbert spaces} \label{S:rkhs}
Let $\mathcal{H}$ be a Hilbert space of analytic functions on the
open unit disc $\mathbb{D}$ which satisfies the following
properties:

\begin{enumerate}[(i)]
\item For any $n\geq 0$, $\chi_n\in\mathcal{H}$, where $\chi_n(z)=z^n$, and
\[
\lim_{n \to \infty} \frac{\|\chi_{n+1}\|_{\mathcal{H}}}{\|\chi_n\|_{\mathcal{H}}} = 1;
\]
\item $(\chi_n/\|\chi_n\|_{\mathcal{H}})_{n \geq 0}$ is a Riesz basis for $\mathcal{H}$.
\end{enumerate}

More explicitly, the last condition says that, for each $f\in\mathcal{H}$, there is a unique sequence of complex numbers $(a_n)_{n\geq 0}$ such that
\[
f=\sum_{n=0}^{\infty}a_n \chi_n,
\]
where the series is norm convergent and
\[
c_1 \, \sum_{n=0}^{\infty}  |a_n|^2 \|\chi_n\|_{\mathcal{H}}^2 \leq \|f\|_{\mathcal{H}}^2 \leq c_2 \, \sum_{n=0}^{\infty} |a_n|^2 \|\chi_n\|_{\mathcal{H}}^2.
\]
The constants $c_1$ and $c_2$ are universal for $\mathcal{H}$ and,
throughout the paper, $c_1$ and $c_2$ refer to these values. In
particular, we are interested in the case $c_1=c_2=1$, i.e. when
$(\chi_n/\|\chi_n\|_{\mathcal{H}})_{n \geq 0}$ is an orthonormal
basis. In that case, $\mathcal H$ is known as a \emph{weighted Hardy
space}, see \cite[page 14]{Cowen} or \cite{Shields74}.

The axioms provided above can also be slightly generalized. We may assume that there is a sequence of strictly positive weights $(w_n)_{n \geq 0}$, with
\begin{equation} \label{E:lim1}
\lim_{n \to \infty} \frac{w_{n+1}}{w_{n}} = 1,
\end{equation}
such that $(\chi_n/\sqrt{w_n})_{n \geq 0}$ is a Riesz basis for $\mathcal{H}$. Hence, for each finitely supported sequence of complex numbers $(a_n)_{n \geq 0}$, we have
\[
c_1 \, \sum_{n=0}^{\infty} w_n |a_n|^2 \leq \left\| \sum_{n=0}^{\infty} a_n \chi_n \right\|_{\mathcal{H}}^{2} \leq c_2 \, \sum_{n=0}^{\infty} w_n |a_n|^2.
\]
In the axioms presented above, we have $w_n = \|\chi_n\|_{\mathcal{H}}^2$. For the simplicity of notations, we will keep $w_n$ throughout the paper.

In the following, the set of all polynomials of degree less or equal to $n$ is denoted by $\mathcal{P}_n$, and the set of all polynomials by $\mathcal{P}$. The family of all functions which are analytic on a disc larger than $\mathbb{D}$ is denoted by $\mbox{Hol}(\overline{\mathbb{D}})$. The above set of axioms have several immediate consequences. We collect these properties in the following.

\begin{enumerate}[$(a)$]
\item
Since $(\chi_n/\sqrt{w_n})_{n \geq 0}$ is a Riesz basis for $\mathcal{H}$, if the infinite sequence $(a_n)_{n \geq 0}$ is such that $\sum_{n=0}^{\infty} w_n |a_n|^2 < \infty$, then $\sum_{n=0}^{\infty} a_n \chi_n \in \mathcal{H}$. In particular, in the light of \eqref{E:lim1}, this happens whenever
\[
\limsup_{n \to \infty} |a_n|^{1/n} < 1.
\]
In the language of function spaces, this means that
$\mbox{Hol}(\overline{\mathbb{D}}) \subset \mathcal{H}$. In fact,
for this property, we only need the part of axiom (i) which says
that $\limsup_{n\to \infty}\frac{w_{n+1}}{w_n}\leq 1$.
\item The previous property also implies that $\mathcal{P}$ is a dense subspace of $\mathcal{H}$.

\item For each $z \in \mathbb{D}$, we have
\begin{eqnarray*}
\left| \sum_{n=0}^{\infty} a_n z^n \right|^2 &\leq& \sum_{n=0}^{\infty} w_n |a_n|^2 \times \sum_{n=0}^{\infty} \frac{|z|^{2n}}{w_n} \\
&\leq& \left( \frac{1}{c_1} \sum_{n=0}^{\infty} \frac{|z|^{2n}}{w_n} \right) \left\| \sum_{n=0}^{\infty} a_n \chi_n \right\|_{\mathcal{H}}^{2}.
\end{eqnarray*}
Note that the series $\sum_n |z|^{2n}w_n^{-1}$ is convergent for any
$z\in\mathbb D$ because $\liminf_{n\to
+\infty}\frac{w_{n+1}}{w_n}\geq 1$. Thus the evaluation functional
\[
\begin{array}{cccc}
\mathcal{H} & \longrightarrow & \mathbb{C}\\
f & \longmapsto & f(z)
\end{array}
\]
is bounded on $\mathcal{H}$ and by the Riesz representation theorem,
there is $k_{z}^{\mathcal{H}} \in \mathcal{H}$, the so called {\em
reproducing kernel} of $\mathcal{H}$, such that
\[
f(z) = \langle f, k_{z}^{\mathcal{H}}\rangle_{\mathcal{H}}, \qquad (f \in \mathcal{H}, \, z \in \mathbb{D}).
\]
In the particular case $c_1=c_2$ (which corresponds to the case when $(\chi_n/\|\chi_n\|_{\mathcal{H}})_{n \geq 0}$ is an orthonormal basis), we have
\[
k_{\lambda}^{\mathcal{H}} =\sum_{n=0}^{\infty} \frac{\bar{\lambda}^n}{\|\chi_n\|_{\mathcal{H}}^{2}} \,\, \chi_n.
\]

\item The shift operator
\[
\begin{array}{cccc}
S: & \mathcal{H} & \longrightarrow & \mathcal{H}\\
 & f & \longmapsto & zf
\end{array}
\]
is well-defined and bounded on $\mathcal{H}$. In fact, it is straightforward to verify that
\begin{equation} \label{E:normS}
\frac{c_1}{c_2} \, \sup_{n \geq 0} \frac{w_{n+1}}{w_n} \leq \|S\|^2_{\mathcal{L}(\mathcal{H})} \leq \frac{c_2}{c_1} \, \sup_{n \geq 0} \frac{w_{n+1}}{w_n}.
\end{equation}
\end{enumerate}

A function $\phi$ is called a {\em multiplier} of $\mathcal{H}$ if
for each $f \in \mathcal{H}$ we have $\phi f \in \mathcal{H}$. The
set of all multipliers of $\mathcal{H}$ (which is an algebra) will
be denoted by $\mathcal{M}(\mathcal{H})$. It is well known that
$\mathcal{M}(\mathcal{H}) \subset \mathcal{H} \cap
H^\infty(\mathbb{D})$, and the multiplication operator
\[
\begin{array}{cccc}
M_\phi : & \mathcal{H} & \longrightarrow & \mathcal{H}\\
 & f & \longmapsto & \phi f
\end{array}
\]
is well-defined and bounded on $\mathcal{H}$. Hence, we equip
$\mathcal{M}(\mathcal{H})$ with the operator norm. The above
observation about the shift operator implies that each polynomial is
indeed a multiplier of $\mathcal{H}$ because for $f\in\mathcal H$
and $p\in\mathcal P$, we have $pf=p(S)f$.

An important example of Hilbert space which fulfills our hypothesis (i) and (ii) is the family of {\em Dirichlet-type spaces} $\mathcal{D}_\alpha$, defined as the space of analytic functions $f(z)=\sum_{n\geq 0}a_n z^n$ on $\mathbb{D}$ satisfying
\[
\|f\|_{\mathcal{D}_\alpha}^2 := \sum_{n=0}^{\infty} (n+1)^\alpha |a_n|^2  < \infty.
\]
This family by itself includes celebrated spaces like the classical
{\em Dirichlet space} ($\alpha = 1$), the {\em Hardy space}
($\alpha=0$) and the {\em Bergman space} ($\alpha=-1$). Therefore,
the results presented below work in particular in all these
classical situations. Note that in \cite{aleman}, Aleman studied
 the cyclicity problem in a class of spaces which fall in the class we
 introduced in this paper.

\subsection{Terminology}

A function $f \in \mathcal{H}$ is said to be \emph{cyclic} if the polynomial multiples of the function form a dense subspace of $\mathcal{H}$. The density of polynomials in $\mathcal{H}$ immediately ensures that the constant function $1$ is cyclic. Therefore, the boundedness of the shift operator on $\mathcal{H}$ yields that an element $f \in \mathcal{H}$ is cyclic if and only if there exists a sequence of polynomials $(p_n)_{n \geq 1}$ such that
\[
\| p_n f -1 \|_{\mathcal{H}} \longrightarrow 0
\]
as $n$ goes to $\infty$. As one of the basic necessary conditions, since the point evaluations are bounded,  a cyclic function cannot have any zeros inside $\mathbb{D}$. That is why in the following we consider functions which live on a disc containing $\mathbb{D}$, with no zeros on $\mathbb{D}$, and then study the effect of its zeros on $\mathbb{T}$ or even outside $\overline{\mathbb{D}}$.

Adopting some concepts from \cite{BCLSS1}, we define the {\em $n$-th optimal norm} and the {\em optimal norm} of $f$ respectively by
\[
\epsilon_n = \min_{p \in \mathcal{P}_n} \|pf-1\|_{\mathcal{H}}
\quad \mbox{and} \quad
\epsilon = \lim_{n \to \infty} \epsilon_n,
\]
where we recall that $\mathcal{P}_n$ denotes the family of all
polynomials of degree less or equal to $n$. Note that
$(\epsilon_n)_n$ is a decreasing sequence, whence, the limit
$\epsilon$ exists and is well-defined.  The {\em $n$-th optimal
approximant} to $1/f$ is any polynomial $p_n^*\in \mathcal{P}_n$
satisfying $\|p_n^* f -1\|_{\mathcal{H}} = \epsilon_n$ (actually, we
will shortly see that $p_n^*$ is uniquely determined). With this new
language, $f \in \mathcal{H}$ is cyclic if and only if its optimal
norm is zero, i.e. $\epsilon = 0$.

\subsection{Statement of results}

As the first step, in Theorem \ref{T1}, we show that optimal
approximants exist and they are uniquely determined as the solution
to a linear system. We dedicate Section \ref{section2} to its proof.

Then, in Section \ref{cyc-hol0}, we study the cyclicity for functions which are holomorphic on a disc bigger than the open unit disc. Let $\mathfrak{H}$ denote the subclass of functions in $\mbox{Hol}(\overline{\mathbb{D}})$ which have no zeros inside $\mathbb{D}$. Our main result, Theorem \ref{T2},  concerns the characterization of cyclicity of the elements of $\mathfrak{H}$. In short, among several other intermediate results, we show that every function $f\in\mathfrak{H}$ is cyclic in $\mathcal{H}$ if and only if
\[
\sum_{k=0}^\infty \frac{1}{w_k} = \infty.
\]

Note that in the case when $\mathcal{H}=\mathcal{D}_\alpha$, then $w_k=(k+1)^{\alpha}$. Thus, Theorem~\ref{T2} says that each function $f\in\mathfrak{H}$ is cyclic in $\mathcal{D}_\alpha$ if and only if $\alpha\leq 1$ and we recover a result of Brown--Shields.

The proof will be based on reducing the problem to studying the function $f(z)=1-z$, in the same spirit as in \cite{BCLSS1}, finding explicitly both the optimal polynomials and optimal norms for this function and then infering the result for other functions from this one.

As in \cite{BCLSS1}, one show that comparable speeds of decay apply to the optimal norms for all functions in $\mathfrak{H}$ whose zeros on the boundary are simple and with at least one boundary zero. For higher multiplicity of the zeros it is not known whether slower decay is possible, although the lower bounds still apply. In Section \ref{section4}, we present some remarks on this, leading to a sharpening of Theorem \ref{T2}, and we introduce two problems on the algebraic properties of the spaces $\mathcal{H}$.

Finally, in Section \ref{section5}, we make some observations on the
distribution of zeros of optimal polynomials and we address some
problems that are left open in this context.

\section{Existence and uniqueness}\label{section2}

The following result ensures that the $n$-th optimal approximant always exists and, moreover, it is uniquely determined.

\begin{thm} \label{T1}
Let $f \in \mathcal{H} \setminus \{0\}$, and let $n \geq 0$. Then the $n$-th optimal approximant to $1/f$ uniquely exists, and is obtained via the solution to the linear system
\[
Ma=b,
\]
where $M=[m_{ij}]$ is a matrix with entries
\[
m_{ij} = \langle \chi_j f , \chi_i f \rangle_{\mathcal{H}}, \qquad  (i,j=0,...,n),
\]
$b=[b_i]$ is a vector with entries
\[
b_i= \langle 1,\chi_i f \rangle_{\mathcal{H}}, \qquad  (i=0,...,n),
\]
and the entries of the vector $a=[a_i]$ are the coefficients of the $n$-th optimal approximant
$p^*_n(z)= \sum_{i=0}^n a_i z^i$.
\end{thm}

\begin{proof}
Existence and uniqueness of the optimal approximant is an easy consequence of the Hilbert structure. Fix $f \in \mathcal{H} \setminus \{0\}$, and consider the set
\[
V_n (f) = \{pf : p \in \mathcal{P}_n\}.
\]
Since each polynomial is a multiplier of $\mathcal{H}$, the collection $V_n(f)$ is a closed (finite dimensional) Hilbert subspace of $\mathcal{H}$. Hence, the  orthogonal projection $\Pi_n$ from $\mathcal{H}$ onto $V_n(f)$ is well-defined. Therefore, there is a unique element $p_n^*f\in V_n$ satisfying $p_n^*f= \Pi_n(1)$ and, by the basic principles of inner product spaces, this element is such that $\|p_n^*f-1\|_{\mathcal{H}} = \epsilon_n$. As $f$ is not identically zero, $p^*_n$ is uniquely determined.

To verify the proposed linear system, note that the optimality of $p^*_n$ is equivalent with $p^*_nf - 1$ being orthogonal to $qf$ for all $qf \in V_n(f)$. This is fulfilled if and only if
\[
\langle p^*_n f -1, \chi_{i}f \rangle_{\mathcal{H}} =0, \qquad (0 \leq i \leq n).
\]
Put the independent term $\langle 1,\chi_{i}f \rangle_{\mathcal{H}}$ on the right hand side.  Then decompose $p^*_n$ as a sum of monomials $p^*_n  = \sum_{j=0}^n a_j \chi_{j}$ and the result follows.

Note that if $(\chi_n)_{n \geq 0}$ is an orthogonal sequence (equivalently, if $c_1=c_2=1$),  then
\[
\langle 1,\chi_{i}f \rangle_{\mathcal{H}} = \langle 1,f \rangle_{\mathcal{H}} \, \delta_{i0} =  w_0\overline{f(0)} \, \delta_{i0},
\]
where $\delta_{i0} = 1$ if $i=0$, and $\delta_{i0}=0$ otherwise.
\end{proof}

\section{Cyclicity in RKHS} \label{cyc-hol0}

Now we are ready to state and prove our main result.

\begin{thm} \label{T2}
The following are equivalent:
\begin{enumerate}[$(a)$]
\item Every function $f\in\mathfrak{H}$ is cyclic in $\mathcal{H}$.
\item There exists $f\in\mathfrak{H}$ such that $f$ is cyclic in $\mathcal{H}$ and, moreover,   $f(e^{i \theta})=0$ at some point $e^{i \theta} \in \TT$.
\item \[
\sum_{k=0}^\infty \frac{1}{w_k} = \infty.
\]
\end{enumerate}
\end{thm}

The proof has several steps.

\subsection{Reductions of the problem}

We can divide our function $f$ into simple pieces and study the cyclicity of each one independently. This action is justified by the following observation.

\begin{lem} \label{lemprodc}
The product of two cyclic multipliers is cyclic.
\end{lem}

\begin{proof}
If $f$ and $g$ are two cyclic multipliers, then there exist
sequences of polynomials $(p_n)_{n \geq 1}$ and $(q_m)_{m \geq 1}$,
such that $p_n f -1 \longrightarrow 0$ and $q_m g -1 \longrightarrow
0$ in the norm of $\mathcal{H}$ as $n$ and $m$ go to $\infty$,
respectively. Fix $\varepsilon>0$, and choose $n$ such that $\|p_n
f-1\|_{\mathcal{H}} \leq \varepsilon$. Then, for this $n$, choose
$m=m(n)$ large enough to guarantee that $\|q_m g-1\|_{\mathcal{H}}
\leq \varepsilon / \|p_n f\|_{\mathcal{M}(\mathcal{H})}$. Now, the
sequence $p_n q_{m(n)}$ of polynomials proves that $fg$ is cyclic.
To see this, by the triangular inequality, we have
\begin{eqnarray*}
\|p_n q_{m(n)} fg - 1\|_{\mathcal{H}} &\leq&  \|p_n q_{m(n)} fg - p_n f\|_{\mathcal{H}} + \|p_n f -1\|_{\mathcal{H}}\\
&\leq&  \|p_n f\|_{\mathcal{M}(\mathcal{H})} \, \|q_{m(n)} g - 1\|_{\mathcal{H}} + \|p_n f -1\|_{\mathcal{H}} \leq 2\varepsilon.
\end{eqnarray*}
\end{proof}

To further reduce the problem, fix $\lambda \in \mathbb{T}$, and consider the mapping $U_\lambda: \mathcal{H} \longrightarrow \mathcal{H}$ defined by
\[
\sum_{n=0}^{\infty} a_n \chi_n \overset{U_\lambda}{\longmapsto} \sum_{n=0}^{\infty} a_n \lambda^n \chi_n.
\]
Since $(\chi_n/\sqrt{w_n})_{n \geq 0}$ is a Riesz basis for $\mathcal{H}$, then $U_\lambda$ is well-defined and bounded operator on $\mathcal{H}$. Moreover, its inverse is $U_{\bar{\lambda}}$ and we have
\begin{equation} \label{Ulambda}
U_{\lambda} S U_{\bar{\lambda}} = \lambda S.
\end{equation}
Note that, $U_\lambda$ is defined such that
\[
(U_{\lambda} f)(z) = f(\lambda z), \qquad (z \in \mathbb{D}).
\]

\begin{lem} \label{rotation}
Let $\lambda \in \mathbb{T}$. If $f \in \mathcal{H}$ is cyclic, then so is $U_\lambda f$.
\end{lem}

\begin{proof}
Let $p(z) = \sum_{k=0}^{n} a_k z^k$. Then, by \eqref{Ulambda},
\begin{eqnarray*}
U_{\lambda} M_p U_{\bar{\lambda}} &=& U_{\lambda} \left( \sum_{k=0}^{n} a_k S^k \right)  U_{\bar{\lambda}}
= \sum_{k=0}^{n} a_k U_{\lambda} S^k U_{\bar{\lambda}}\\
&=& \sum_{k=0}^{n} a_k (U_{\lambda} S U_{\bar{\lambda}})^k
= \sum_{k=0}^{n} a_k \lambda^k S^k = M_q,
\end{eqnarray*}
where $q$ is the polynomial $q(z) = \sum_{k=0}^{n} a_k \lambda^k z^k$. Since $f$ is cyclic, there is a sequence $(p_j)_{j \geq 1}$ such that $p_jf \longrightarrow 1$ in $\mathcal{H}$.  Therefore, we also have
\[
U_{\lambda} M_{p_j} f \overset{\mathcal{H}}{\longrightarrow} U_{\lambda}1=1
\]
as $j \to \infty$. Since $U_{\lambda} M_{p_j} = M_{q_j} U_{\lambda}$, we deduce
\[
M_{q_j} U_{\lambda} f \overset{\mathcal{H}}{\longrightarrow} 1.
\]
This means that $U_{\lambda} f$ is cyclic in $\mathcal{H}$.

\end{proof}

Given $f\in\mathfrak{H}$, we write it as
\begin{equation} \label{E:zero-decomp}
f(z) = g(z) \prod_{k=1}^{n} (1-\bar{\lambda}_kz)
\end{equation}
where $\lambda_k \in \mathbb{T}$ and $g$ has no zeros on $\overline{\mathbb{D}}$. Therefore, in the light of lemmas \ref{lemprodc} and \ref{rotation}, the problem reduces to characterizing the cyclicity of $f(z)=1-z$ and functions with no zeros in the closed unit disc.

\subsection{No zeros on $\overline{\mathbb{D}}$}

We show that if there is no zeros on $\overline{\mathbb{D}}$, then everything works well. This entails to calculating the spectral radius of $S$ which is, by itself, an interesting result. We remind that $\sigma(S)$ and $\rho(S)$ represent respectively the spectrum and the spectral radius of $S$.

\begin{lem} \label{Lemfinmulit0}
We have $\rho(S)=1$.
\end{lem}

\begin{proof}
Equation \eqref{E:normS} extends to
\[
\frac{c_1}{c_2} \, \sup_{n \geq 0} \frac{w_{n+k}}{w_n} \leq \|S^k\|^2_{\mathcal{L}(\mathcal{H})} \leq \frac{c_2}{c_1} \, \sup_{n \geq 0} \frac{w_{n+k}}{w_n} , \qquad (k \geq 1).
\]
The assumption \eqref{E:lim1} now implies that
\[
\rho(S)=\lim_{k \to \infty} \|S^k\|_{\mathcal{L}(\mathcal{H})}^{1/k} = 1.
\]
\end{proof}

\begin{rem}
In fact, we can go further and completely find $\sigma(S)$. Let $\mathcal{H}_w$ be the space $\mathcal{H}$ equipped with the new norm
\[
\left\|\sum_{k=0}^{\infty} a_k z^k \right\|_{\mathcal{H}_w}^2 = \sum_{k=0}^{\infty} w_k |a_k|^2.
\]
By our main assumptions, we have
\[
c_1 \| f \|_{\mathcal{H}_w}^2 \leq \|f\|_{\mathcal{H}}^2 \leq c_2 \|f\|_{\mathcal{H}_w}^2.
\]
Hence, the operator $S$ is similar to the multiplication operator $S_w \in \mathcal{L}(\mathcal{H}_w)$. Now, it is easy to verify that
\[
(S_w^* f)(z) = \sum_{k=0}^{\infty} \frac{w_{k+1}}{w_k} \, a_{k+1} z^k, \qquad (f \in \mathcal{H}_w).
\]
In particular, with each $\lambda \in \mathbb{D}$ and
\[
k_{\lambda}^{\mathcal{H}_w}(z) = \sum_{k=0}^{\infty} \frac{\bar{\lambda}^k}{w_k} \, z^k,
\]
then $k_\lambda^{\mathcal{H}_w}\in\mathcal{H}_w$ and for any $f\in\mathcal{H}$
\[
f(\lambda)=\langle f,k_{\lambda}^{\mathcal{H}_w} \rangle_{\mathcal{H}_w}.
\]
Therefore, we have $S_w^* k_{\lambda}^{\mathcal{H}_w} = \bar{\lambda} k_{\lambda}^{\mathcal{H}_w}$ and thus $\sigma(S)=\sigma(S_w) = \overline{\mathbb{D}}$.
\end{rem}

It was elementary to see that $\mbox{Hol}(\overline{\mathbb{D}}) \subset \mathcal{H}$ (see Section \ref{S:rkhs}, part (a)). However, with some extra care, we can show that $\mbox{Hol}(\overline{\mathbb{D}}) \subset \mathcal{M}(\mathcal{H})$. This observation is crucial in what follows.

\begin{lem} \label{Lemfinmulit}
$\mbox{Hol}(\overline{\mathbb{D}}) \subset \mathcal{M}(\mathcal{H})$.
\end{lem}

\begin{proof}
By Lemma \ref{Lemfinmulit0}, $\rho(S)=1$. Therefore, if $(a_k)_{k \geq 0}$ is any sequence such that
\begin{equation} \label{coefles1}
\limsup_{k \to \infty} |a_k|^{1/k} < 1,
\end{equation}
then
\[
T := \sum_{k=0}^{\infty} a_k S^k \in \mathcal{L}(\mathcal{H}).
\]
In particular, if $\phi \in \mbox{Hol}(\overline{\mathbb{D}})$, then
we can write $\phi(z) = \sum_{k=0}^{\infty} a_k z^k$, where the
Taylor coefficients $(a_k)_{k \geq 0}$ satisfy \eqref{coefles1}.
Now, it remains to note that if $f\in\mathcal H$, then
\[
Tf=\sum_{k=0}^{\infty} a_k S^k f=\sum_{k=0}^{\infty} a_k z^k f=\phi f,
\]
and thus $\phi$ is a multiplier of $\mathcal{H}$.
\end{proof}

Finally, we need the following simple criteria for cyclicity which is a standard result.

\begin{lem} \label{Lemfinmulit00}
Let $f \in \mathcal{M}(\mathcal{H})$ and $1/f \in \mathcal{H}$. Then $f$ is cyclic in $\mathcal{H}$.
\end{lem}

\begin{proof}
For any polynomial $p$ we have
\[
\|pf-1\|_{\mathcal{H}} = \|f(p-1/f)\|_{\mathcal{H}} \leq \|f\|_{\mathcal{M}(\mathcal{H})} \|p-1/f\|_{\mathcal{H}}.
\]
Since polynomials are dense in $\mathcal H$, the last sum can be made arbitrarily small. Hence, $f$ is cyclic in $\mathcal{H}$.
\end{proof}

\begin{thm} \label{T4}
Let $f \in \mbox{Hol}(\overline{\mathbb{D}})$, with no zeros on $\overline{\mathbb{D}}$. Then $f$ is cyclic in $\mathcal{H}$.
\end{thm}

\begin{proof}
If $f$ is analytic on a disc bigger than the open unit disc and, moreover, has no zeros on $\bar{\mathbb{D}}$, then $1/f$ has the same property and thus, by Lemma \ref{Lemfinmulit}, both $f$ and $1/f$ belong to $\mathcal{M}(\mathcal{H})$. In particular,  $1/f \in \mathcal{H}$. Therefore, by Lemma \ref{Lemfinmulit00}, $f$ is cyclic in $\mathcal{H}$.
\end{proof}

\subsection{At least one zero on $\mathbb{T}$}
We now deal with the remaining problem, which is the heart of Theorem \ref{T2}. What remains is to characterize the cyclicity of the function $f(z)=1-z$.

\begin{thm} \label{T3}
Let $f(z)=1-z$. Then the $n$-th optimal norm of $f$ satisfies
\[
c_1 \leq \epsilon^2_n \sum_{k=0}^{n+1} \frac{1}{w_k} \leq c_2
\] and these rates are achieved by the polynomials given by
\[
p^*_n(z)= \sum_{k=0}^n \left( 1- \frac{\,\, \sum_{j=0}^k \frac{1}{w_j} \,\,}{\sum_{j=0}^{n+1} \frac{1}{w_j}} \right) z^k.
\]
In particular, if $(\chi_n/\sqrt{w_n})_{n \geq 0}$ is an orthonormal
basis of  $\mathcal{H}$, then the $n$-th optimal norm of $f$ is
precisely
\[
\epsilon_n = \left(\sum_{k=0}^{n+1} \frac{1}{w_k} \right)^{-1/2}
\]
and the described polynomials are actually the optimal approximants.
\end{thm}

\begin{proof}

In the case of an orthogonal basis, to check that $p_n^*$ is the
optimal approximant could be reduced to checking that it satisfies
the corresponding linear system in Theorem \ref{T1}. The lower bound
for the optimal norm could be easily deduced from the optimality of
these polynomials. While the above approach is feasible, we present
a different proof that works for any Riesz basis.

 If $p(z) = \sum_{k=0}^{n} a_k z^k$, then
\[
(1-z)p(z)-1 = \sum_{k=0}^{n+1} (a_{k}-a_{k-1})z^k,
\]
with conventions $a_{-1}=1$ and $a_{n+1}=0$. Since
\[
\sum_{k=0}^{n+1} (a_{k}-a_{k-1}) = -1,
\]
the Cauchy-Schwarz inequality implies
\begin{eqnarray*}
1 &\leq& \left(\sum_{k=0}^{n+1} |a_{k}-a_{k-1}| \right)^2\\
&\leq& \sum_{k=0}^{n+1} w_k |a_{k}-a_{k-1}|^2  \times \sum_{k=0}^{n+1} \frac{1}{w_k} \\
&\leq& \frac{1}{c_1}  \, \| (1-z)p(z)-1 \|^2_{\mathcal{H}}  \times \sum_{k=0}^{n+1} \frac{1}{w_k}.
\end{eqnarray*}
Therefore, for each $p \in \mathcal{P}_n$,
\begin{equation} \label{E:eps1}
\|pf-1 \|^2_{\mathcal{H}} \geq \frac{c_1}{\sum_{k=0}^{n+1} \frac{1}{w_k}}.
\end{equation}

On the other hand, with the special choice of $a_k$ suggested in the theorem, we have
\[
a_k -a_{k-1} = \frac{\,\, -\frac{1}{w_k} \,\,}{\sum_{j=0}^{n+1} \frac{1}{w_j}}, \qquad (0 \leq k \leq n+1).
\]
Therefore,
\begin{eqnarray}
\| pf-1 \|^2_{\mathcal{H}} &\leq& c_2 \sum_{k=0}^{n+1} w_k |a_{k}-a_{k-1}|^2 \notag\\
&\leq& \frac{c_2}{ \sum_{k=0}^{n+1} \frac{1}{w_k}} . \label{E:eps2}
\end{eqnarray}
Therefore, by \eqref{E:eps1} and \eqref{E:eps2}, the bounds for the norms follow.

\end{proof}

Passing to the limit, Theorem \ref{T3} immediately implies the following result.

\begin{cor}\label{corap}
Let $f(z)=1-z$. Then the optimal norm of $f$ satisfies
\[
\frac{c_1} {\sum_{k=0}^{\infty} \frac{1}{w_k}} \leq \epsilon^2 \leq \frac{c_2}{\sum_{k=0}^{\infty} \frac{1}{w_k}}.
\]
In particular, if $(\chi_n/\sqrt{w_n})_{n \geq 0}$ is an orthonormal
basis of $\mathcal{H}$, then the optimal norm of $f$ is precisely
\[
\epsilon = \left(\sum_{k=0}^{\infty} \frac{1}{w_k} \right)^{-1/2}.
\]
\end{cor}
In particular, Corollary~\ref{corap} says that $f(z)=1-z$ is cyclic in $\mathcal{H}$ if and only if
\begin{equation}\label{cns-cyclicite-10}
\sum_{k=0}^{\infty}\frac{1}{w_k}=\infty.
\end{equation}

\subsection{Proof of Theorem \ref{T2}}
Now, we have all necessary tools to establish the proof of Theorem \ref{T2}.

\noindent $(a) \Longrightarrow (b)$: Trivial.

\noindent $(b) \Longrightarrow (c)$: We decompose $f$ as in
\eqref{E:zero-decomp}, and write $f=gp$, where the polynomial $p$
has all its zeros on $\mathbb{T}$ and has at least one zero there.
Since $f$ is cyclic, there is a sequence of polynomials $(p_n)_{n
\geq 1}$ such that $fp_n \longrightarrow g$ in $\mathcal{H}$. But,
by Lemma \ref{Lemfinmulit}, $1/g \in  \mathcal{M}(\mathcal{H})$.
Therefore, $pp_n \longrightarrow 1$ in $\mathcal{H}$. Consider one
of the factors of $p$, say $1-\bar{\lambda}_kz$. The previous
relation means that $1-\bar{\lambda}_kz$ is cyclic. Hence, by Lemma
\ref{rotation}, $1-z$ is cyclic and then Corollary \ref{corap}
ensures that \eqref{cns-cyclicite-10} holds.

\noindent $(c) \Longrightarrow (a)$: Assume that
\eqref{cns-cyclicite-10} holds. Let $f$ be a function in
$\mbox{Hol}(\overline{\mathbb{D}})$ which has no zeros inside
$\mathbb{D}$. The function $f$ has only a finite number of zeros on
$\mathbb{T}$ (by the isolated zeros principle). Then, we can
decompose $f$ as in \eqref{E:zero-decomp}. By Theorem \ref{T4}, the
function $g$ is cyclic in $\mathcal{H}$.  Moreover, we can write
\[
1- \lambda_k z = U_{\lambda_k} (1-z), \qquad (\lambda_k \in \mathbb{T}).
\]
As it was discussed in Corollary \ref{corap}, under the  assumption \eqref{cns-cyclicite-10}, the function $z \longmapsto 1-z$ is cyclic. Moreover, the cyclicity is invariant under rotation (Lemma \ref{rotation}). Hence, each $1- \lambda_k z$ is cyclic and then, by Lemma~\ref{Lemfinmulit} and Lemma~\ref{lemprodc}, $f$ is cyclic.

\section{Sharp rates and algebraic properties} \label{section4}

For the case when $g \in \mathfrak{H}$, and the boundary zeros of $f$ are simple (i.e., multiplicity one), we can actually sharpen the proof of Theorem \ref{T2} to show that the optimal norms for $g$ are comparable to those for the particular function $z \longmapsto1-z$. Any function in $\mathfrak{H}$ with at least one boundary zero has $n$-optimal norms decaying not faster than those for $z \longmapsto1-z$. To see that there is also an upper control, define $\phi_H(n) = \sum_{k=0}^n \frac{1}{\omega_k}$ and denote by $\mathcal{A}(\TT)$, the analytic Wiener algebra, that is, the space of analytic functions for which the Taylor coefficients are absolutely summable, with norm of $f(z)= \sum_{n=0}^{\infty} a_n z^n$ given by
\[
\|f\|_{\mathcal{A}(\TT)}= \sum_{n=0}^{\infty} |a_n|.
\]
Then we have the following result.

\begin{thm} \label{T:ptop=t}
Let $f \in \mathfrak{H}$ with all its zeros on the boundary of the unit disc being simple zeros. Then there is a sequence of polynomials $(p_n)_{n \geq 1}$ such that the sequence $(p_nf)_{n \geq 1}$ is uniformly bounded in the $\mathcal{A}(\TT)$-norm and
\[
\|p_nf-1\|^2 \leq \frac{C}{\phi_\mathcal{H}(n)},
\]
where $C>0$ is a constant depending on $f$ and $\mathcal{H}$, but not on $n$.
\end{thm}

The proof is omitted since it works in the same way as that of \cite[Proposition 3.2]{BCLSS1}. The only minor change is the statement of the  \cite[Lemma 3.3]{BCLSS1}, which should read as follows. Denote by $\hat{h}(n)$, the Taylor coefficient of degree $n$ for the function $h$.

\begin{lem}
Suppose $f$ is a polynomial of degree $t$ with no zeros inside the
disc
 and with all its zeros on the boundary being simple zeros.
 If $n>t$, then there is a constant $C=C(\mathcal{H},f)$ such that
\[
\left|\sum_{k=0}^{n} \phi_\mathcal{H}(k) \widehat{1/f}(k) \hat{f}(n-k) \right| \leq \frac{C}{\omega_n}, \qquad (n \geq 0).
\]
\end{lem}

\subsection{Products of functions}
An interesting open question is to find an estimate on the optimal norms for higher multiplicity. It seems natural to ask whether these results are true without the assumptions on the multiplicity of the zeros. This would be true if we could give a positive answer to the following open problem, which is interesting by itself.

\begin{prob}
Is $\mathcal{H} \cap \mathcal{A}(\TT)$ an algebra?
\end{prob}
The answer to this question is positive whenever $\{\omega_n\}_{n \geq 0}$ is a bounded sequence, since then $\mathcal{A}(\TT) \subset \mathcal{H}$. It also holds whenever $\mathcal{H}$ is an algebra (trivially), and when $\mathcal{H}$ is the Dirichlet space. To see this last claim, observe that if $f$ and $g$ are bounded Dirichlet functions, then we have $$\|fg\|^2_D \leq |f(0)g(0)|^2 + \|f\|^2_{H^{\infty}} \|g\|^2_D + \|g\|^2_{H^{\infty}} \|f\|^2_D.$$

\subsection{Quotients of functions}
The following, somewhat related, question is settled in \cite{aleman} for some special weights which are the moments of a function in $L^2(\mathbb{T})$. However, the question makes sense in our general setting.

\begin{prob}
Can every function $f \in \mathcal{H}$ be written as the quotient of
two bounded functions in $\mathcal{H}$, i.e.
\[
f = \frac{g}{h}, \qquad (g,h \in \mathcal{H} \cap H^\infty)?
\]
\end{prob}

\section{Asymptotic distribution of the zeros of the optimal polynomials}\label{section5}

In \cite{BCLSS1}, it was pointed out that the distribution of the zeros of the optimal polynomials should contain some information about the cyclicity of the function. A classical result of this type is the Enestr\"om theorem (1893) which restricts the region where the zeros of a polynomial may lie, in terms of its coefficients. See, for instance, \cite{ASV1}.

\begin{thm}[Enestr\"om]
Let $p(z)= \sum_{k=0}^n a_k z^k$, where $a_k >0$ for $0 \leq k \leq n$.
Then all the zeros of $p$ lie in the annulus $\{\alpha \leq |z|\leq \beta \}$, where
\[
\alpha= \min_{0\leq k <n} \frac{a_k}{a_{k+1}}
\qquad \mbox{and} \qquad
\beta= \max_{0\leq k <n} \frac{a_k}{a_{k+1}}.
\]
\end{thm}

Suppose now that the monomials form an orthogonal basis in the
space. Applying this result to the optimal polynomials for
$f(z)=1-z$ (see Theorem \ref{T3}) implies the following conclusion.

\begin{cor}
The zeros of $p_n^*$ lie in the region
\[
\left\{ \min_{0 \leq k <n}\omega_{k+1} \sum_{j=k+2}^{n+1} \frac{1}{\omega_j}  \leq \frac{1}{ |z| - 1} \leq \max_{0 \leq k <n} \omega_{k+1} \sum_{j=k+2}^{n+1} \frac{1}{\omega_j} \right\}.
\]
\end{cor}

It would be interesting to sharpen this result or to solve the following problem at least in the case when $f(z)=1-z$.
\begin{prob}
Find the asymptotic distribution of the zeros of the optimal approximants to $1/f$, in terms of the function $f$ and the space $\mathcal{H}$.
\end{prob}

However, a much more natural condition arises when we apply Enestr\"om's result to find points where $p_n^* f -1$ is zero, which is, in fact, a function that should be close to zero uniformly on the compact subsets of the disc.
\begin{cor}
Let $f(z)=1-z$ and $p_n^*$ the corresponding optimal polynomial of degree $n$. Then the zeros of $p_n^* f -1$ lie in the region
\[
\left\{ \min_{0 \leq k <n}\frac{\omega_{k+1}}{\omega_k} \leq |z| \leq \max_{0 \leq k <n}\frac{\omega_{k+1}}{\omega_k} \right\}.
\]
\end{cor}


\begin{thebibliography}{1}
\bibitem{aleman}
A. Aleman, {\em Hilbert spaces of analytic functions between the Hardy and the Dirichlet space}, Proc. AMS {\bf 115} (1992), 97--104.

\bibitem{ASV1}
N. Anderson, E. B. Saff and R. S. Varga, {\em On the Enestr\"om-Kakeya theorem and its sharpness}, Linear Algebra Appl. {\bf 28} (1979), 5--16.

\bibitem{ARSW1}
N. Arcozzi, R. Rochberg, E. T. Sawyer, and B. D. Wick, {\em The
Dirichlet space: a survey}, New York Math. J. {\bf 17A} (2011), 45--86.

\bibitem{BCLSS1}
C. B\'en\'eteau, A. A. Condori, C. Liaw, D. Seco, and
A. A. Sola, {\em Cyclicity in Dirichlet-type spaces and extremal polynomials},
J. Anal. Math. (to appear).

\bibitem{B1}
A. Beurling, {\em On two problems concerning linear operators in Hilbert space}, Acta Math. {\bf 81} (1949), 239--255.

\bibitem{B2}
L. Brown, {\em Invertible elements in the Dirichlet space}, Canad.
Math. Bull. {\bf 33} (1990), 419--422.

\bibitem{BC1}
L. Brown and W. Cohn, {\em Some examples of cyclic vectors in
the Dirichlet space}, Proc. Amer. Math. Soc. {\bf 95} (1985), 42--46.

\bibitem{BS1}
L. Brown and A. L. Shields, {\em Cyclic vectors in the Dirichlet space},
Trans. Amer. Math. Soc. {\bf 285} (1984), 269--304.

\bibitem{Cowen}
C. Cowen  and B. MacCluer, {\em Composition operators on spaces of analytic functions}, Studies in Advanced Mathematics,
1995.

\bibitem{EFKR1}
O. El-Fallah, K. Kellay, and T. Ransford,
{\em Cyclicity in the Dirichlet space},
Ark. Mat. {\bf 44} (2006), 61--86.

\bibitem{EFKR2}
O. El-Fallah, K. Kellay, and T. Ransford, {\em On the Brown-Shields
conjecture for cyclicity in the Dirichlet space}, Adv. Math. {\bf 222} (2009),
2196--2214.



\bibitem{HS1}
H. Hedenmalm and A. Shields, {\em Invariant subspaces in Banach spaces
of analytic functions}, Michigan Math. J. {\bf 37} (1990), 91--104.



\bibitem{RS1}
S. Richter and C. Sundberg, {\em Multipliers and invariant subspaces
in the Dirichlet space}, J. Operator Theory {\bf 28} (1992), 167--186.

\bibitem{R1}
W. T. Ross, {\em The classical Dirichlet space}, in {\em Recent
advances in operator-related function theory}, Contemp. Math. {\bf 393} (2006),
171--197.

\bibitem{Shields74}
A. Shields, {\em Weighted shift operators and analytic function
theory}, in {\em Topics in operator theory}, Math. Surveys {\bf 13}
(1974), 49--128.

\end{thebibliography}
\end{document}